\documentclass[11pt]{amsart}

\usepackage{amsfonts}
\usepackage{amssymb, latexsym, amsthm, verbatim,stmaryrd}
\usepackage{indentfirst}
\usepackage{amsmath}

\allowdisplaybreaks

\numberwithin{equation}{section}
\theoremstyle{definition}
\newtheorem{Thm}[equation]{Theorem}
\newtheorem{Prop}[equation]{Proposition}
\newtheorem{Cor}[equation]{Corollary}
\newtheorem{Lem}[equation]{Lemma}

\newtheorem{Rmk}[equation]{Remark}
\newtheorem*{Que*}{Question}

\makeatletter
\def\imod#1{\allowbreak\mkern5mu{\operator@font mod}\,\,#1}
\makeatother

\begin{document}

\title[Half-Integral Weight Modular Forms]{Half-Integral Weight Modular Forms and Modular forms for Weil representations}
\author[Yichao Zhang]{Yichao Zhang}
\address{Department of Mathematics, Harbin Institute of Technology, Harbin, China 150001}
\email{yichao.zhang@hit.edu.cn}
\date{}

\begin{abstract}
We establish an isomorphism between certain complex-valued and vector-valued modular form spaces of half-integral weight, generalizing the well-known isomorphism between modular forms for $\Gamma_0(4)$ with Kohnen's plus condition and modular forms for the Weil representation associated to the discriminant form for the lattice with Gram matrix $(2)$. With such an isomorphism, we prove the Zagier duality and write down the Borcherds lifts explicitly.
\end{abstract}

\maketitle

\section*{Introduction}
\noindent

The theory of modular forms is of fundamental importance in many parts of modern number theory and many other related fields. The weakly holomorphic modular forms, namely those with possible poles at cusps, received less attention than the holomorphic ones. One of the few exceptions is the modular $j$-function, which processes representation-theoretic and arithmetic properties. Things changed when Borcherds, in his seminal papers \cite{borcherds1995automorphic} and \cite{borcherds1998automorphic}, constructed a multiplicative theta lifting, also known as Borcherds automorphic product, which sends weakly holomorphic vector-valued modular forms of full level to modular forms in the form of infinite products for orthogonal groups. The Borcherds' lift is in general a meromorphic modular form in the form of an infinite product, and Borcherds' theory shows precisely the location of its divisors.

Remarkably, in his work on traces of singular moduli, Zagier \cite{zagier1998traces} proved a duality for Fourier coefficients of modular forms weights $k$ and $2-k$ for six small half-integral $k$, with which he gave a different proof of Borcherds's theorem for $\Gamma_0(4)$. Such duality is now known as Zagier duality and Zagier dualities for various types of modular forms, of integral or half-integral weight, have been proved since then (see \cite{zhang2015zagier} for a list of reference on this research).

Many important works have been built on Borcherds lifts by directly employing vector-valuded modular forms ever since. However, the full-level and vector-valued condition is not convenient to work with. To overcome such difficulty, in the case of integral weights, Bruinier and Bundschuh \cite{bruinier2003borcherds} constructed an isomorphism between prime-level complex-valued modular forms and full-level vector-valued modular forms. Such an isomorphism proves to be useful. With such an isomorphism, Bruinier \cite{bruinier2014converse} proved a converse theorem for Borcherds' lift, Bruinier and Yang \cite{bruinier2007twisted} constructed twisted Borcherds products, Rouse \cite{rouse2006zagier} and Choi \cite{choi2006simple} proved the Zagier duality for level $5,13,17$, and Kim and Lee \cite{kim2012rank} provided automorphic corrections to some rank two hyperbolic Kac-Moody algebras. Later, following the work of Scheithauer (\cite{scheithauer2009weil} and \cite{scheithauer2011some}), the author generalized Bruinier and Bundschuh's isomorphism, proved the Zagier duality for general level, and with Kim and Lee, provided more automorphic corrections(see \cite{zhang2015isomorphism}, \cite{zhang2015zagier}, \cite{kim2015weakly}).

On the other hand, it has been known for a long time that the half-integral weight modular forms for $\Gamma_0(4)$ can be sent to full-level vector-valued modular forms with two components, which can be made into an isomorphism if we choose the Kohnen's plus space (see \cite{eichler1985theory} for example). With such an isomorphism, Borcherds' theorem in \cite{borcherds1995automorphic} is merely a special case of his automorphic product theorem in \cite{borcherds1998automorphic}. One expects that such an isomorphism can be constructed as in the case of half-integral weight, which is the main purpose of this paper. 

The construction and the proof are similar to the integral-weight case and we need the concrete formula for the matrix coefficients of the Weil representations obtained by Str\"omberg \cite{stromberg2013weil}. We actually work a little more generally and treat \emph{transitive} discriminant forms. For example, the $p$-component $D_p$ of the discriminant form $D$ can be $(\mathbb Z/p\mathbb Z)^2$, or more precisely $D_p=p^{\pm 2}$ with $\pm =-\left(\frac{-1}{p}\right)$. We shall mainly consider the case when the level of $D$ is $N=4M$ with $M$ odd and square-free. The Atkin-Lehner operators $Y(p)$ and $Y(4)$ decompose the space of modular forms into common eigenspaces: the Kohnen's plus condition corresponds to an eigenvalue of $Y(4)$ as expected; for a component $p^{\pm 2}$, the eigenvalue $-1$ is chosen; for a component $p^{\pm 1}$, the eigenvalues give the sign vector condition. Such treatment shows some similarity between $Y(4)$ and $Y(p)$ with $\chi_p\neq 1$. With the isomorphism established, we see that for components $p^{\pm 2}$, the eigenvalue $-1$ gives the \emph{new} space in the sense that it is mapped to full-level (vector-valued) modular forms. This provides a characterization of the eigenspaces of the $Y(p)$ operator when $\chi_p\neq 1$, which is missing from the literature (see \cite{ueda1993twisting}). Since we consider weakly holomorphic modular forms and the Petersson inner product is no longer available, this sheds light on what forms are new as the Kohnen's plus condition does. Similar result should also hold and be easier for the integral-weight case.

After establishing the isomorphism, we move on to prove the corresponding Zagier duality and write down the Borcherds theorem in the case of $\textrm{O}(2,1)$ explicitly. For simplicity, we shall assume that the components $p^{\pm 2}$ do not appear. By introducing \emph{reduced modular forms} and employing the obstruction theorem of Borcherds \cite{borcherds1999gross} as we did in \cite{zhang2015zagier}, we can prove the Zagier duality (Theorem \ref{R-Thm-1}) without much effort. For the Borcherds' theorem (Theorem \ref{Thm2}), most of the parts are straightforward except the computation of the Weyl vector. By checking through the proofs, we note that our isomorphism above can be applied to meromorphic or real analytic modular forms, hence in particular to Zagier's non-holomorphic modular form $\mathbf G$ of weight $3/2$. The corresponding vector-valued modular forms $\mathbf G_N$ for $\mathbf G$ satisfies Lemma 9.5 and Corollary 9.6 of \cite{borcherds1998automorphic}, from which the explicit formula for the Weyl vector follows.

Here is the layout of this paper: after providing the basics and fixing the notations in Section 1, we consider and classify transitive discriminant forms in Section 2. In Section 3, we briefly cover the Atkin-Lehner operators and the corresponding eigenspaces. In Section 4, we assume that $D_2=2^+_{\pm 1}$ and describe the $\epsilon$-condition that is needed for the isomorphism. We then establish the isomorphism in Section 5. In Section 6, we assume further that $D_p=p^{\pm 1}$ for all odd $p\mid N$, prove Zagier duality and translate Borcherds' theorem. Finally we construct some examples in the last section.


\section{Preliminaries}\label{Preliminaries}
\noindent
We recall the basics on discriminant forms, modular forms of half-integral weight and modular forms for Weil representations. For more details on discriminant forms, one may consult \cite{conway1998sphere}, \cite{nikulin1980integral}, or \cite{scheithauer2009weil}.

\subsection{Discriminant forms}
A discriminant form is a finite abelian group $D$ with a quadratic form $q: D\rightarrow \mathbb Q/\mathbb Z$, such that the symmetric bilinear form defined by $(\beta,\gamma)=q(\beta+\gamma)-q(\beta)-q(\gamma)$ is nondegenerate, namely, the map $D\rightarrow \text{Hom}(D,\mathbb Q/\mathbb Z)$ defined by $\gamma\mapsto (\gamma,\cdot)$ is an isomorphism. We shall also write $q(\gamma)=\frac{\gamma^2}{2}$. We define the level of a discriminant form $D$ to be the smallest positive integer $N$ such that $Nq(\gamma)=0$ for each $\gamma\in D$. It is well-known that if $L$ is an even lattice then $L'/L$ is a discriminant form, where $L'$ is the dual lattice of $L$. Conversely, any discriminant form can be obtained this way. With this, the signature $\text{sign}(D)\in \mathbb Z/8\mathbb Z$ is defined to be the signature of $L$ modulo $8$ for any even lattice $L$ such that $L'/L=D$.

Every discriminant form can be decomposed into a direct sum of Jordan $p$-components for primes $p$ and each Jordan $p$-component can be written as a direct sum of indecomposible Jordan $q$-components with $q$ powers of $p$. Such decompositions are not unique in general. To fix our notations, we recall the possible indecomposible Jordan $q$-components as follows.

Let $p$ be an odd prime and $q>1$ be a power of $p$. The indecomposible Jordan components with exponent $q$ are denoted by $q^{\delta_q}$ with $\delta_q=\pm 1$; it is a cyclic group of order $q$ with a generator $\gamma$, such that $q(\gamma)=\frac{a}{q}$ and $\delta_q=\left(\frac{2a}{p}\right)$. These discriminant forms both have level $q$.

If $q>1$ is a power of $2$, there are also precisely two indecomposable \emph{even} Jordan components of exponent $q$, denoted $q^{\delta_q 2}$ with $\delta_q=\pm 1$; it is a direct sum of two cyclic groups of order $q$, generated by two generators $\gamma$, $\gamma'$, such that if $\delta_q=1$, we have
\[q(\gamma)=q(\gamma')=0, \quad (\gamma,\gamma')=\frac{1}{q},\] and if $\delta_q=-1$, we have
\[q(\gamma)=q(\gamma')=\frac{1}{q},\quad (\gamma,\gamma')=\frac{1}{q}.\]
Such components have level $q$. There are also \emph{odd} indecomposable Jordan components in this case, denoted by $q_t^{\pm 1}$ with $\pm 1=\left(\frac{2}{t}\right)$ for each $t\in \left(\mathbb Z/8\mathbb Z\right)^\times$. Explicitly, $q^{\pm 1}_t$ is a cyclic group of order $q$ with a generator $\gamma$ such that $q(\gamma)=\frac{t}{2q}$. Clearly, these discriminant forms have level $2q$.

To give a finite direct sum of indecomposable Jordan components of the same exponent $q$, we multiply the signs, add the ranks, and add all subscripts $t$ ($t=0$ if there is no subscript). So in general, the $q$-component of a discriminant form is given by $q_t^{\delta_q n}$ ($t=0$ if $q$ is odd or the form is even). Set $k=k(q^{\delta_q n}_t)=1$ if $q$ is not a square and $\delta_q=-1$, and $0$ otherwise. If $q$ is odd, then define $p$-excess$(q^{\pm n})=n(q-1)+4k\imod 8$, and if $q$ is even, then define oddity$(q^{\pm n}_t)=2$-excess$(q^{\pm n}_t)=t+4k\imod 8$.

Let $D$ be a discriminant form and assume that $D$ has a Jordan decomposition $D=\oplus_q q^{\delta_q n_q}_t$ where the sum is over distinct prime powers $q$. Then
\[p\text{-excess}(D)=\sum_{q: p\mid q} p\text{-excess}(q^{\delta_q n_q}_t).\]
We recall the well-known \emph{oddity formula}:
\[\text{sign}(D)+\sum_{p>2}p\text{-excess}(D)=\text{oddity}(D)\imod 8.\]

\subsection{Metaplectic covers}
Throughout this note unless otherwise stated, $k\in \frac{1}{2}+\mathbb Z$ and $\mathbb H$ denotes the upper half plane. Let $\textrm{Mp}_2^+(\mathbb R)$ be the metaplectic cover of $\textrm{GL}_2^+(\mathbb R)$ (see \cite[Section 1]{shimura1973modular}), so a typical element in $\textrm{Mp}_2^+(\mathbb R)$ is of the form $(A,\phi)$ where $\phi$ is a holomorphic function on $\mathbb H$ and
\[A=\begin{pmatrix}
a&b\\c&d
\end{pmatrix}\in \textrm{GL}_2^+(\mathbb R),\quad  \phi(\tau)=tj(A,\tau),\text{ for some } t\in\mathbb C, |t|=1.\] Here $j(A,\tau)=\textrm{det}(A)^{-\frac{1}{4}}(c\tau+d)^\frac{1}{2}$ is the square-root of the usual one for the integral-weight case.
The group multiplication is defined by
\[(A,\phi)(B,\psi):=(AB,\phi(B\tau)\psi(\tau)),\quad (A,\phi), (B,\psi)\in \textrm{Mp}_2^+(\mathbb R).\] We follow \cite{shimura1973modular} and introduce the multiplier system $\nu$ on $\Gamma_0(4)$:
\[\nu(A)=\left(\frac{c}{d}\right)\varepsilon_d^{-1},\quad A=\begin{pmatrix}
a&b\\c&d
\end{pmatrix}\in\Gamma_0(4),\]
where $\left(\frac{c}{d}\right)$ is the Kronecker symbol. We recall that if $d$ is odd and positive, this is the usual Jacobi symbol; $\left(\frac{c}{d}\right)=0$ if $\textrm{gcd}(c,d)>1$; $\left(\frac{c}{d}\right)=\textrm{sign}(c)\left(\frac{c}{-d}\right)$ if $cd\neq 0$; $\left(\frac{2}{d}\right)=\left(\frac{d}{2}\right)$ and $\left(\frac{0}{\pm 1}\right)=\left(\frac{\pm 1}{0}\right)=1$. These conditions determine the symbol by complete multiplicativity in $d$. Note that \[\overline{\nu}(A)=\nu^3(A)=\left(\frac{-1}{d}\right)\nu(A),\quad \nu(A)\nu(A^{-1})=1,\quad A\in\Gamma_0(4).\] For any $A\in\textrm{GL}_2^+(\mathbb R)$, we will denote
\[\tilde A=(A,j(A,\tau))\in\textrm{Mp}_2^+(\mathbb R).\]
Moreover, if
$A\in\Gamma_0(4)$, we denote \[A^*=(A,\nu(A)j(A,\tau)).\]
It is well-known that $A\mapsto A^*$ gives an injective homomorphism and we denote its image by $\Gamma^*_0(4)$. The image of $\Gamma_0(N)$ for $4\mid N$ will be denoted by $\Gamma^*_0(N)$.

Let $\textrm{Mp}_2(\mathbb Z)$ be the metaplectic double cover of $\textrm{SL}_2(\mathbb Z)$ inside $\textrm{Mp}_2^+(\mathbb R)$, consisting of pairs $(A,\phi)$ with $A=\begin{pmatrix}
a&b\\c&d
\end{pmatrix}\in \textrm{SL}_2(\mathbb Z)$  and  $\phi^2=c\tau +d$. Let $S$ and $T$ denote the standard generators of $\textrm{SL}_2(\mathbb Z)$, so in $\text{Mp}_2(\mathbb Z)$,
\[
\tilde S=\left(\begin{pmatrix}
0&-1\\
1&0
\end{pmatrix},\sqrt{\tau}\right),\quad \tilde T=\left(\begin{pmatrix}
1&1\\
0&1
\end{pmatrix},1\right),\]
generate $\text{Mp}_2(\mathbb Z)$.
We shall also need
\[Z:=\widetilde{-I}=\left(\begin{pmatrix}
-1&0\\
0&-1
\end{pmatrix},i\right),\]
and we have $Z^4=\tilde I$ and $\tilde S^2=Z$.

\subsection{Modular forms}
Let $(A,\phi)\in\text{Mp}_2^+(\mathbb R)$ and $f$ be a function on $\mathbb H$. The weight-$k$ slash operator is defined by
\[\left(f|_k(A,\phi)\right)(\tau)=\phi^{-2k}(\tau)f(A\tau), \quad A=\begin{pmatrix}a&b\\c&d\end{pmatrix}.\]
We shall usually drop the weight $k$ from this notation.

Given any discriminant form $D$, let $r$ denote the signature of $D$ and Let $\{\mathfrak e_\gamma: \gamma\in D\}$ be the standard basis of the group algebra $\mathbb C[D]$. The Weil representation $\rho_D$ attached to $D$ is a unitary representation of $\text{Mp}_2(\mathbb Z)$ on $\mathbb C[D]$ such that
\begin{eqnarray*}
\rho_D(\tilde T)\mathfrak e_\gamma&=&e(q(\gamma))\mathfrak e_\gamma,\\
\rho_D(\tilde S)\mathfrak e_\gamma&=&\frac{i^{-\frac{r}{2}}}{\sqrt{|D|}}\sum_{\beta\in D}e(-(\beta,\gamma))\mathfrak e_\beta,
\end{eqnarray*}
where $e(x)=e^{2\pi i x}$ and $|D|$ is the order of $D$. In particular, we have $\rho_D(Z)\mathfrak e_\gamma=i^{-r}\mathfrak e_{-\gamma}$. For convenience, we shall also denote $e_m(x)=e^{\frac{2\pi ix}{m}}$.

Denote by $\text{Aut}(D)$ the automorphism group of $D$, that is, the group of group automorphisms of $D$ that preserve the norm (or the quadratic form). The action of elements in $\text{Aut}(D)$ and that of $\rho_D$ commute on $\mathbb C[D]$. We caution here that our $\rho_D$ is the same as that in \cite{borcherds1998automorphic} and \cite{bruinier2003borcherds}, but conjugate to the one used in \cite{scheithauer2009weil} and \cite{scheithauer2011some}.

We denote by $\mathcal A(k,\rho_D)$ the space of functions $F=\sum_{\gamma\in D}F_\gamma \mathfrak e_\gamma$ on $\mathbb H$, valued in $\mathbb C[D]$, such that
\begin{itemize}
\item $F|A:=\sum_\gamma F_\gamma |A \mathfrak e_\gamma=\rho_D(A) F$ for all $A\in\text{Mp}_2(\mathbb Z)$,
\item $F$ is holomorphic on $\mathbb H$ and meromorphic at $\infty$; namely, for each $\gamma\in D$, $F_\gamma$ is holomorphic on $\mathbb H$ and has Fourier expansion at $\infty$ with at most finitely many negative power terms.
\end{itemize}
More explicitly, if $F=\sum_\gamma F_\gamma\in \mathcal A(k,\rho_D)$, then
\[F_\gamma(\tau)=\sum_{n\in q(\gamma)+\mathbb Z, n\gg -\infty}a(\gamma,n)q^n.\] Denote by $\mathcal M(k,\rho_D)$ and $\mathcal S(k,\rho_D)$ the subspace of holomorphic modular forms and the subspace of cusp forms, respectively. Because the action of $\text{Aut}(D)$ and that of $\rho_D$ commute on $\mathbb C[D]$, the vector-valued modular forms will be invariant under $\textrm{Aut}(D)$.
We define $\mathcal A^{\text{inv}}(k,\rho_D)$ to be the subspace of functions that are invariant under $\text{Aut}(D)$. By the action of $Z$ and $F_\gamma=F_{-\gamma}$ for $F\in\mathcal A(k,\rho_D)$, we must have $2k\equiv r\imod 4$, since $\gamma\mapsto -\gamma$ defines an element in $\text{Aut}(D)$. Therefore, we shall always assume that $2k\equiv r\imod 4$. Similarly, we define $\mathcal M^{\text{inv}}(k,\rho_D)$ and $\mathcal S^\text{inv}(k,\rho_D)$.

For each Dirichlet character $\chi$ of modulo $N$, the space $\mathrm{A}(N,k,\chi)$ consists of holomorphic functions $f$ on $\mathbb H$ such that $f|A^*=\chi(A)f$ for each $A\in\Gamma_0(N)$ and $f$ is meromorphic at cusps. By considering $A=-I$, we see that for the space to be nonzero, we necessarily have $\chi(-1)=1$. The subspace of holomorphic forms and that of cuspforms are denoted by $\mathrm{M}(N,k,\chi)$ and $\mathrm{S}(N,k,\chi)$ respectively.

Each discriminant form $D$ can be decomposed uniquely into $p$-components  $D=\oplus_pD_p$ and each Dirichlet character $\chi$ can also be decomposed uniquely into $p$-components $\chi=\prod_p\chi_p$. For each positive integer $m$, we shall denote $D_m=\oplus_{p\mid m}D_p$ and $\chi_m=\prod_{p\mid m}\chi_p$ for convenience.

\section{Transitive Discriminant Forms}

\noindent
We call a discriminant form $D$ \emph{transitive} if for any $n\in\mathbb Q/\mathbb Z$, the action of $\textrm{Aut}(D)$ is transitive on the subset of elements of norm $n$.
In this section, we classify transitive discriminant forms in general and prove some of their properties.

It is trivial that if $D$ is transitive then $D$ is anisotropic. We classify transitive discriminant forms in the following proposition.
\begin{Prop}\label{T-Prop-1}
A discriminant form $D$ is transitive if and only if $D=\oplus_pD_p$ such that
\begin{itemize}
	\item for an odd prime $p$, $D_p$ is either trivial or equal to $p^{\pm 1}$, or $D_p=p^{+2}$ when $p\equiv 3\imod 4$, or $D_p=p^{-2}$ when $p\equiv 1\imod 4$;
	\item $D_2$ is either trivial or equal to one of the following:
	\[2^{+3}_{\pm 3},\quad 2^{+2}_{\pm 2},\quad 2^{+1}_{\pm 1},\quad 2^{-2}, \quad 4^{\pm 1}_t,\quad 4^{\pm 1}_t\oplus 2^{+1}_{+1}.\]
\end{itemize}
\end{Prop}
\begin{proof}
We first show that if $D=\oplus_pD_p$ is transitive, then any indecomposable component is equal to either $p^{\pm 1}$ for some odd prime $p$ or one of the following: $2^{+1}_{\pm 1}$, $2^{-2}$, $4^{\pm 1}_t$. If $p$ is an odd prime and $q=p^f$ with $f\geq 2$ and assume that $q^{\delta_q}$ appears in $D$ as an indecomposable component. Then consider the element $\gamma=p^{f-1}+q\mathbb Z\in \mathbb Z/q\mathbb Z$. It is easy to see that $q(\gamma)=0$, so $D$ is isotropic and $D$ is not transitive. The claim on the $2$-components follows similarly.

We then show that if $D=\oplus_pD_p$ is transitive, then $D_p$ has the form in the statement. We have just seen that for an odd prime $p$, $D_p=p^{\delta_pn_p}$ for some $\delta_p\in\{\pm 1\}$ and $n_p\geq 0$. Therefore, $D_p$ is a quadratic space over $\mathbb F_p$. If $n_p\geq 3$, by \cite[Chapter I, Theorem 3, Corollary 2]{serre1973course}, $D_p$ is isotropic. If $n_p=2$ and $p\equiv 1\imod 4$, then $-1\imod p$ is a square. By \cite[Chapter IV, Proposition 3', Corollary 2]{serre1973course}, since $D$ is anisotropic, the two indecomposable components represent different elements. So the two components have different signs and $\delta_p=-1$. The case when $p\equiv -1\imod 4$ follows in the same way. When $p=2$, with some concrete computation, we can also proceed similarly.

Finally we show that such $D$ in the statement is transitive. Since $\textrm{Aut}(D)=\bigoplus_{p}\textrm{Aut}(D_p)$, we only have to prove that each $D_p$ is transitive. When $D_p=p^{\delta n_p}$ for an odd $p$, since $D_p$ is anisotropic, any two nonzero elements $\gamma,\gamma'$ with the same norm generate two non-degenerate one-dimensional $\mathbb F_p$-subspaces, so $\gamma\mapsto \gamma'$ extends to an element in $\textrm{Aut}(D_p)$ by Witt's theorem. This shows that such $D_p$ is transitive. That those $D_2$ in question are also transitive follows from explicit computation.
\end{proof}

\begin{Lem}\label{T-Lem-5}
	Let $p$ be an odd prime and $D=p^{\pm 2}$ be transitive. Then $D$ represents all elements in $\frac{1}{p}\mathbb Z/\mathbb Z-\{0\}$.
\end{Lem}
\begin{proof}
By replacing $\frac{1}{p}\mathbb Z/\mathbb Z$ with $\mathbb Z/p\mathbb Z$, $D$ is a quadratic form over $\mathbb F_p$. To prove that $D$ represents $\left(\mathbb Z/p\mathbb Z\right)^\times$, for any $\alpha\in \left(\mathbb Z/p\mathbb Z\right)^\times$, we consider the quadratic form  $q'(x,y,z)=q(x,y)\oplus -\alpha z^2$. By \cite[Chapter I, Theorem 3, Corollary 2]{serre1973course}, $q'$ is universal, hence represents $0$. Since $q$ does not represents $0$, it follows that $q$ represents $\alpha$.
\end{proof}

The following lemma is crucial in proving our isomorphism later and we prove it in a similar way as in \cite{zhang2015isomorphism} with minor modifications.

\begin{Lem}\label{T-Lem-4}
Let $D$ be transitive and for a fixed modular form $F=\sum_{\gamma}F_\gamma\mathfrak e_\gamma\in\mathcal A^\text{inv}(k,\rho_D)$, let $W=\textrm{span}_\mathbb{C}\left\{F_\gamma\colon \gamma\in D\right\}$. Then

(1) $W=\textrm{span}_\mathbb{C}\left\{F_0|A\colon A\in \textrm{Mp}_2(\mathbb Z)\right\}$.
In particular, if $F_0=0$, then $F=0$.

(2) If $f=\sum_{\gamma\in D}a_\gamma F_\gamma$ is $T$-invariant, then $f=a_0F_0$.
\end{Lem}
\begin{proof}
Denote $W'$ the space spanned by $F_0|A$,  $A\in \textrm{Mp}_2(\mathbb Z)$ and we need to prove $W=W'$. Note that $F|A=\rho_D(A)F$, so
\[F_0|A=(\rho_D(A)F,\mathfrak e_0)=\sum_{\gamma}F_\gamma(\rho_D(A)\mathfrak e_\gamma,\mathfrak e_0).\]
This implies that $W'\subset W$.

To prove the other inclusion, assume that $\sum_{\gamma\in D}a_\gamma F_\gamma\in W'$ for $a_\gamma\in\mathbb C$. We claim that for each $\gamma$ with $q(\gamma)=n$ and $a_n:=\sum_{\beta\colon q(\beta)=n}a_\beta\neq 0$, then $F_\gamma\in W'$. In particular, for any subset $C\subset D$, if $\sum_{\gamma\in S}F_\gamma\in W'$, then $F_\gamma\in W'$ for each $\gamma\in C$.
To prove the claim, we can rewrite $\sum_{\gamma}a_\gamma F_\gamma=\sum_{n\imod \mathbb Z}F_n$, with $F_n=\sum_{\gamma\colon q(\gamma)=n}a_\gamma F_\gamma$. Since $F$ is invariant under $\text{Aut}(D)$ and $D$ is transitive, we have $F_n=a_nF_\gamma$. Since $a_n\neq 0$, we only have to prove that $F_n\in W'$. Now the transformation rule of $F$ under $T$ shows that $F_n|\tilde T=e(n)F_n$. Since $W'$ is invariant under the action of $\text{Mp}_2(\mathbb Z)$ and $\sum_{n}F_n\in W'$, we have $\sum_ne(nj)F_n\in W'$ for each positive integer $j$. Since $e(n)$'s are distinct mutually, this implies that $F_n\in W'$ by the theory of Vandermonde matrix.

To prove that $W\subset W'$, we only have to show that $F_\gamma\in W'$ for each $\gamma\in D$.
Assume that $D$ has level $N$ and we choose
\[A=\left(\begin{pmatrix}
0&-1\\1& N
\end{pmatrix},\sqrt{\tau+N}\right)\in \textrm{Mp}_2(\mathbb Z).\]
By \cite[Theorem 6.4]{stromberg2013weil}, we have that for each $\gamma\in D$,
\[(\rho_D(A)\mathfrak e_\gamma,\mathfrak e_0)=\xi(0,1)|D|^{-\frac{1}{2}}\neq 0,\]
which is independent of $\gamma$. It follows that $\sum_{\beta\in D}F_\beta\in W'$, so $F_\gamma\in W'$ for each $\gamma$ because of the claim. This complete the proof for part (1).

Similarly, if $f=\sum_{\gamma\in D}a_\gamma F_\gamma=\sum_{n\imod \mathbb Z}F_n$ is $T$-invariant, we only have to prove that $F_n=0$ if $n\neq 0$. Clearly $T$-invariant functions form a subspace, so by the same argument using Vandermonde matrix, we must have $F_n$ is $T$-invariant. This forces $F_n=0$ if $n\neq 0$. We are done with part (2).
\end{proof}

From now on, we always assume that $D$ is transitive. Therefore, the level $N$ of $D$ will be of the following form: $N_p=1$ or $p$ for an odd prime $p$ and $N_2=1,2,4$ or $8$. In other words, $N$ is the conductor of a quadratic Dirichlet character. We shall assume that $N$ is of this form throughout this paper. Note that $ND=\{0\}$ and $|D|$ and $N$ share the same set of prime divisors, but one may not divide the other in general.

\begin{Rmk}
We need Lemma \ref{T-Lem-4} (1) to show that $F_0$ determines $F$, but the assumption that $D$ is transitive may not be necessary. For example, if $D=2^{+2}$, that $F_0$ determines $F$ still holds, in which case Lemma \ref{T-Lem-4} (2) is no longer true (see \cite[Example 13.7]{borcherds1998automorphic}).
\end{Rmk}

\section{Atkin-Lehner Operators}

In this section, we consider the Atkin-Lehner operators on the space $\mathrm{A}(N,k,\chi)$, where $N=4M$, $M$ is odd square-free, $k\in\frac 12+\mathbb Z$, and $\chi$ is a quadratic Dirichlet character modulo $N$. Even though the half-integral weight case is similar to the integral weight case, we treat them in details for later computation. The main references are \cite{kohnen1982newforms} and \cite{ueda1993twisting}.

For any odd divisor $m$ of $N$, we choose $\gamma_m$ and $\gamma_{4m}$ in $\textrm{SL}_2(\mathbb Z)$ such that
\[\gamma_m\equiv \left\{\begin{array}{ll}
S & \imod m^2,\\
I & \imod (\frac N m)^2,
\end{array}
\right.\qquad
\gamma_{4m}:=S\gamma^{-1}_{M/m}\equiv \left\{\begin{array}{ll}
S & \imod (4m)^2,\\
I & \imod (\frac Mm)^2.
\end{array}
\right.
\]
The existence of $\gamma_m$ follows from the existence of such matrices in $\textrm{SL}_2(\mathbb Z/N^2\mathbb Z)$ by Chinese remainder theorem and then from the surjectivity of $\textrm{SL}_2(\mathbb Z)\rightarrow \textrm{SL}_2(\mathbb Z/N^2\mathbb Z)$. We shall also assume for simplicity that all of the entries of $\gamma_m$ are positive; this can be achieved by left and/or right multiplication by matrices in $\Gamma(N^2)$.

For any nonzero integer $m$, let
\[\delta_m=\begin{pmatrix}
m&0\\0&1
\end{pmatrix},\quad \widetilde{\delta_m}=\left(\begin{pmatrix}
m&0\\0&1
\end{pmatrix}, m^{-\frac 14}\right).\]
For any odd positive divisor $m$ of $N$, let $W(m)=\gamma^*_m\widetilde{\delta_m}$.
which makes sense since $m$ is odd and $\gamma_m\in\Gamma_0(4)$. Define
\[\tau_N=\left(1,\sqrt{-i}\right)\widetilde{\beta_N}=\left(\begin{pmatrix}
0&-1\\N&0
\end{pmatrix}, N^{\frac{1}{4}}(-i\tau)^{\frac{1}{2}}\right),\quad \beta_N=\begin{pmatrix}
0&-1\\N&0
\end{pmatrix}.\]
For each divisor $m$, even or odd, of $N$, define $U(m)$ as follows:
\[f|U(m)=m^{\frac k2-1}\sum_{j\imod m}f|\widetilde{\delta_m}^{-1}\tilde T^j=m^{\frac k2-1}\sum_{j\imod m}f|\widetilde{\delta_m^{-1}T^j}.\]
Finally, we define $Y(p)$ for each odd $p\mid N$ by
\[f|Y(p)=p^{1-\frac k2}f|U(p)W(p),\]
and $Y(4)$ by
\[f|Y(4)=4^{1-\frac k2}f|U(4)W(M)\tau_N.\]

We collect a few properties of these operators in the following proposition:
\begin{Prop}\label{A-Prop-1} Let $f\in \mathrm{A}(N,k,\chi)$.

(1) $f|\tau_N\in \mathrm{A}(N,k,\chi\left(\frac N\cdot\right))$ and $f|\tau_N^2=f$.

(2) For each $m\mid N$, $f|U(m)\in \mathrm{A}(N,k,\chi\left(\frac m\cdot\right))$ and $f|U(m)=m^{\frac k2-1}f|[\Delta_1(N)\widetilde{\delta_m}^{-1}\Delta_1(N)]$.

(3) For each $m\mid M$, $f|W(m)\in \mathrm{A}(N,k,\chi\left(\frac m\cdot\right))$ and \[f|W(m)^2=\varepsilon_m^{-2k}\chi_m(-1)\chi_{N/m}(m)f.\]
Moreover, if $m,m'\mid M$ and $(m,m')=1$, then $f|W(m)W(m')=\chi_{m'}(m)f|W(mm')$.

(4) For any $m,m'\mid M$ with $(m,m')=1$, then
\[f|W(m)U(m')=\chi_{m}(m')f|U(m')W(m)\text{ and } f|U(4)W(m)=f|W(m)U(4).\]

(5) For any $m\mid M$, $f|\tau_NU(m)W(m)=\chi_m(M/m)f|W(m)U(m)\tau_N$.
\end{Prop}
\begin{proof}
The first two parts are contained in Proposition 1.4 and 1.5 of \cite{shimura1973modular}. The part (3) and (4) are partially contained in Proposition 1.18 and 1.20 of \cite{ueda1993twisting}. For the second identity in part (4),
\begin{align*}
f|U(4)W(p) & = 4^{\frac k2-1}f|[\Delta_1\widetilde{\delta_4^{-1}}\Delta_1]W(p)
=4^{\frac k2-1}f|[\Delta_1\widetilde{\delta_4^{-1}}W(p)\Delta_1]\\
&=4^{\frac k2-1}f|[\Delta_1\widetilde{\delta_4^{-1}}\gamma_p^*\widetilde{\delta_p}\Delta_1]= 4^{\frac k2-1}f|[\Delta_1(\delta_4^{-1}\gamma_p\delta_4)^*
\widetilde{\delta_p}\widetilde{\delta_4^{-1}}\Delta_1]\\
&=4^{\frac k2-1} f|[\Delta_1\alpha^*W(p)\widetilde{\delta_4^{-1}}\Delta_1],
\end{align*}
where $\alpha=\delta_4^{-1}\gamma_p\delta_4\gamma_p^{-1}\in \Gamma_0(N)$. Since the right-lower entry $d_\alpha$ of $\alpha$ satisfies $d_\alpha\equiv 4\imod p$ and $d_\alpha\equiv 1\imod N/p$, we have $f|\alpha^*=f$. It follows that
\[f|U(4)W(p)=4^{\frac k2-1}f|W(p)[\Delta_1\widetilde{\delta_4^{-1}}\Delta_1]=f|W(p)U(4).\]
The general case follows from part (3) by decomposing $W(m)$.

The proof of part (5) is similar:
\begin{align*}
f|\tau_NU(m)W(m) & = m^{\frac k2-1}f|[\Delta_1\tau_N\widetilde{\delta_m^{-1}}W(m)\Delta_1]\\
&=m^{\frac k2-1}f|[\Delta_1(\tau_N\widetilde{\delta_m^{-1}}W(m)\tau_N^{-1})\tau_N\Delta_1]\\
&=m^{\frac k2-1}f|[\Delta_1\alpha^*\tau_N\Delta_1]\\
&= m^{\frac k2-1}f|(\alpha\gamma_m^{-1})^*[\Delta_1\gamma_m^*\tau_N\Delta_1]\\&
=f|(\alpha\gamma_m^{-1})^*W(m)U(m)\tau_N,
\end{align*}
where
\[\alpha\gamma_m^{-1}\equiv \left\{\begin{array}{cl}
\begin{pmatrix}
(N/m)^{-1} & 0\\
0 & N/m
\end{pmatrix} & \imod m\\
\begin{pmatrix}
1 & 0\\
0 & 1
\end{pmatrix} & \imod N/m\\
\end{array}
\right.\]
(the inverse is taken in $\mathbb Z/m\mathbb Z$). Then we have \[f|\tau_NU(m)W(m)=\chi_m(M/m)f|W(m)U(m)\tau_N,\] as expected.
\end{proof}

Now we consider the operators $Y(p)$ and $Y(4)$.

\begin{Prop}\label{A-Prop-2}
(1) The space $\mathrm{A}(N,k,\chi)$ decomposes under $Y(4)$ into eigenspaces \[\mathrm{A}(N,k,\chi)=\mathrm{A}(N,k,\chi)_{\mu_2^+}\oplus \mathrm{A}(N,k,\chi)_{\mu_2^-}\] where the eigenvalues are \[\mu_2^+=\chi_2(-1)^{k+1}(-1)^{\lfloor \frac{2k+1}{4}\rfloor}2^{\frac 32},\quad \mu_2^-=-2^{-1}\mu_2^+.\] Moreover, $f=\sum_na(n)q^n\in \mathrm{A}(N,k,\chi)_{\mu_2^+}$ if and only if
\[a(n)=0 \text{\ \ whenever }\chi_2(-1)(-1)^{k-\frac 12}n\equiv 2,3\imod 4.\]

(2) Assume that $p\mid M$ with $\chi_p=1$, the space $\mathrm{A}(N,k,\chi)$ decomposes under $Y(p)$ into eigenspaces \[\mathrm{A}(N,k,\chi)=\mathrm{A}(N,k,\chi)_{\mu_p^+}\oplus \mathrm{A}(N,k,\chi)_{\mu_p^-}\] where the eigenvalues are $\mu_p^+=\varepsilon_p^{-1}p^{\frac 12}$, $\mu_p^-=-\mu_p^+$. Moreover, $f=\sum_na(n)q^n\in \mathrm{A}(N,k,\chi)_{\mu_p^{\pm}}$ if and only if
\[a(n)=0 \text{\ \ whenever }\left(\frac np\right)=\mp.\]

(3) Assume that $p\mid M$ with $\chi_p=\left(\frac \cdot p\right)$, the space $\mathrm{A}(N,k,\chi)$ decomposes under $Y(p)$ into eigenspaces \[\mathrm{A}(N,k,\chi)=\mathrm{A}(N,k,\chi)_{\mu_p^+}\oplus \mathrm{A}(N,k,\chi)_{\mu_p^-}\] where the eigenvalues are $\mu_p^+=-1$, $\mu_p^-=p$.

(4) The operators $Y(p)$, $p\mid M$, and $Y(4)$ map $\mathrm{A}(N,k,\chi)$ to itself and they commute mutually. In particular, $\mathrm{A}(N,k,\chi)$ decomposes into direct sum of common eigenspaces for these operators.
\end{Prop}
\begin{proof} That $Y(p_1)$ and $Y(p_2)$ commute is given in Proposition 1.24 of \cite{ueda1993twisting}, part (1) in Proposition 1 of \cite{kohnen1982newforms}, and part (2) and (3) in Proposition 1.27 and 1.29 of \cite{ueda1993twisting}.

We only have to prove that $Y(4)$ and $Y(p)$ commute. That is, we need to prove that
\[f|U(4)W(M)\tau_NU(p)W(p)=f|U(p)W(p)U(4)W(M)\tau_N.\] Indeed, we have $f|U(4)W(M)\in \mathrm{A}(N,k,\chi\left(\frac M\cdot\right))$, and by Proposition \ref{A-Prop-1} (5),
\begin{align*}
f|U(4)W(M)\tau_NU(p)W(p) & = \chi_p(M/p)\left(\frac{M/p}{p}\right)f|U(4)W(M)W(p)U(p)\tau_N.
\end{align*}
Therefore, we only have to prove that
\[\chi_p(M/p)\left(\frac{M/p}{p}\right)f|U(4)W(M)W(p)U(p)=f|U(p)W(p)U(4)W(M).\]
By decomposing $M=p\cdot \frac Mp$ and applying Proposition \ref{A-Prop-1} (3) and (4), this can be done. One should pay attention to the change of characters when applying Proposition \ref{A-Prop-1}.
\end{proof}

It follows that the space $\mathrm{A}(N,k,\chi)$ is a direct sum of common eigenspaces for the operators $Y(4)$, $Y(p)$, $p\mid M$. The subspace $\mathrm{A}(N,k,\chi)_{\mu_2^+}$ is the well-known \emph{Kohnen's plus space} \cite{kohnen1982newforms}. Now let us consider the special case when $\chi_p=1$ for each $p\mid M$. If we denote $Z(p)=\varepsilon_pp^{-\frac 12}Y(p)$ for $p\mid M$, then $Z(p)$ are involutions on $\mathrm{A}(N,k,\chi)_{\mu_2^+}$. For each sign vector $\epsilon=(\epsilon_p)_{p\mid M}$, we denote $\mathrm{A}^\epsilon(N,k,\chi)$ to be the subspace in $\mathrm{A}(N,k,\chi)_{\mu_2^+}$ of modular forms $f$ with $f|Z(p)=\epsilon_pf$ for all $p\mid M$.

\begin{Cor}\label{A-Cor-1}
Assume that $\chi_p=1$ for each $p\mid M$ and let $Z(p)$ be as above. If $f=\sum_na(n)q^n\in \mathrm{A}(N,k,\chi)_{\mu_2^+}$, then $f=\sum_\epsilon f^\epsilon$ with $f^\epsilon=\sum_nb_\epsilon(n)q^n\in \mathrm{A}^\epsilon(N,k,\chi)$, where for $(n,N)=1$,
\[b^\epsilon(n)=2^{-\omega(M)}a(n)\prod_{p\mid M}\left(1+\epsilon_p\left(\frac np\right)\right).\]
\end{Cor}
\begin{proof}
It is clear that
\[f^\epsilon=2^{-\omega(M)}f|\prod_{p\mid M}(1+\epsilon_pZ(p)),\]
and the corollary follows from the proof of Proposition 1.29 of \cite{ueda1993twisting}.
\end{proof}

\section{Discriminant Forms and the $\epsilon$-Condition}

Let $D$ be a transitive discriminant form of odd signature $r$ and level $N$. We first see that it determines a quadratic Dirichlet character $\chi$ mod $N$.

\begin{Lem}\label{D-Lem-1}
A transitive discriminant form $D$ of odd signature $r$ and level $N$ determines an even  quadratic Dirichlet character $\chi$ mod $N$, such that
\[\rho_D(A)\mathfrak e_0=
\nu(A)^{r}\chi(d)\mathfrak e_0,
\quad \text{ for }A=\begin{pmatrix}
a&b\\c&d
\end{pmatrix}\in\Gamma_0(N).\]
Explicitly, if we write $\chi=\prod_p\chi_p$ into $p$-components, then if $p$ is odd,
\[\chi_p(d)=\left\{
\begin{array}{cl}
1, & \quad p\nmid |D|\text{ or } p^2\mid |D|,\\
\left(\frac{d}{p}\right), & \quad \text{otherwise}.
\end{array}
\right.\]
For $\chi_2$, assuming $2\mid |D|$, 
\[\chi_2(d)=\left\{
\begin{array}{cl}
1, & \quad \left(\frac{-1}{|D|}\right)=+1, D_2=2^{+3}_{\pm 3},2^{+1}_{\pm 1},\\
\left(\frac{-4}{d}\right), & \quad \left(\frac{-1}{|D|}\right)=-1, D_2=2^{+3}_{\pm 3},2^{+1}_{\pm 1},\\
\left(\frac{2}{d}\right), & \quad \left(\frac{-1}{|D|}\right)=+1, D_2=4^{+1}_{\pm 1},4^{-1}_{\pm 3},\\
\left(\frac{-2}{d}\right), & \quad \left(\frac{-1}{|D|}\right)=-1, D_2=4^{+1}_{\pm 1},4^{-1}_{\pm 3}.
\end{array}
\right.\]
\end{Lem}
\begin{proof}
That $D$ determines a quadratic Dirichlet character $\chi$ follows from \cite[Lemma 5.6]{stromberg2013weil} and we have the explicit formula
\[\chi(d)=\left\{
\begin{array}{rl}
\left(\frac{d}{2|D|}\right), & \quad \left(\frac{-1}{|D|}\right)=1,\\
\left(\frac{d}{2|D|}\right)\left(\frac{-1}{d}\right), &  \quad \left(\frac{-1}{|D|}\right)=-1,
\end{array}
\right.\]
by elementary computation and by the relation
	\[\textrm{oddity}(D)\equiv \textrm{sign}(D)+\left(\frac{-1}{|D|}\right)-1\imod 4.\]
The formulas for local components $\chi_p$ then follow easily.
\end{proof}

In order to introduce the $\epsilon$-condition, we investigate the representing behavior of $D$.

\begin{Lem}\label{D-Lem-2}
Let $D=\oplus_pD_p$ be of level $N=\prod_pN_p$ and denote $N_p'=N/N_p$. Then for any integer $n$, $D$ represents $\frac{n}{N}$ if and only if $D_p$ represents $\frac{N_p'n}{N_p}$ for each $p\mid N$.
\end{Lem}
\begin{proof}
If $q(\gamma)=\frac{n}{N}$, then we see that for each $p\mid N$, $\gamma_p:=N_p'\gamma\in D_p$ and $g_p(\gamma_p)=q(\gamma_p)=\frac{N_p'n}{N_p}$. Conversely, since for each $p\mid N$, since $\textrm{gcd}(N_p,N_p')=1$, there exist $a_p,b_p\in\mathbb Z$ such that $a_pN_p+b_pN_p'=1$. Now assume $\gamma_p\in D_p$ with $q(\gamma_p)=\frac{N_p'n}{N_p}$ and let $\gamma=\sum_pb_p\gamma_p\in D$. It is clear that
\[q(\gamma)=\left(\sum_pb_p^2N_p'^2\right)\frac{n}{N}.\]
Since $\sum_pb_p^2N_p'^2\equiv 1\imod N_p$ for each $p$, we have $q(\gamma)=\frac{n}{N}$.
\end{proof}

To simplify the exposition, we shall assume that $D_2=2^{+1}_{\pm 1}$ for the rest of this paper. The oddity formula says,
\[\chi_2(-1)=\chi_M(-1)=e_4(r-t)=\left(\frac{-1}{|D|}\right).\] Other cases should be similar but may be more complicated. Now we explain how the data for the two sides, vector-valued and scalar-valued,  correspond as follows.

We begin with $D$, and it determines $N=4M$ with $M$ odd square-free and an even $\chi$. We construct a sign vector $\epsilon=(\epsilon_p)_{p}$ over $p=2$ or $p\mid M$ such that $\chi_p\neq 1$ as follows: if $D_p=p^{\delta_p}$ with $\delta_p=\pm 1$, then we define $\epsilon_p=\chi_p(2M/p)\delta_p$; let $\epsilon_2=t\left(\frac{-1}{N}\right)$ if $D_2=2^{+1}_{t}$. Therefore, $D$ determines $(N,\chi\left(\frac N\cdot\right),\epsilon)$: $N=4M$, $\chi$ an even Dirichlet character modulo $N$ and $\epsilon$ a sign vector. Conversely, given such a triple $(N,\chi\left(\frac N\cdot\right),\epsilon)$, for $p=2$ or $p$ dividing the conductor of $\chi$, we can reconstruct $D_p$ reversely from $\epsilon$ and also for other $p\mid N$ by $D_p=p^{\pm 2}$ where the sign is determined uniquely by the transitivity.

Following Shimura \cite{shimura1973modular}, we shall denote $\chi'=\chi\left(\frac N\cdot\right)$ from now on.  The following lemma describes the set of norms $q(D)$ from the data $(N,\chi',\epsilon)$.
\begin{Lem}\label{D-Lem-3}
Let $D$ and $(N,\chi\left(\frac N\cdot\right),\epsilon)$ correspond. For an integer $n$,  $\frac{n}{N}\imod 1\in q(D)$ if and only if $n\equiv 0$ or $\epsilon_2\imod 4$ and $\chi_p(n)=0$ or $\epsilon_p$ for each odd prime divisor $p$ of the conductor $\chi$.
\end{Lem}
\begin{proof}
Let $c$ denote the odd part of the conductor of $\chi$. By Lemma \ref{D-Lem-2}, $D$ represents $n/N$ if and only if $D_p$ represents $\frac{nN/p}{p}$ for each odd $p\mid N$ and $D_2$ represents $\frac{nN/4}{4}$. Since if $p\nmid 2c$ and $p\mid N$, $D_p$ represents everything by Lemma \ref{T-Lem-5}, we only have to consider the primes $p\mid 2c$. If $p\mid c$, then $D_p=p^{\delta_p}$ represents $\frac{nN/p}{p}$ if and only if $p\mid n$ or $\chi_p(2nN/p)=\delta_p$, hence if and only if $\chi_p(n)=0$ or $\epsilon_p$. Finally, $D_2=2_{t}^{+1}$ represents $\frac{nN/4}{4}$ if and only if $\frac{nN}{4}\equiv 0$ or $t\imod 4$, hence if and only if $n\equiv 0$ or $\epsilon_2\imod 4$. This finishes the proof.
\end{proof}

Given any data $(N,\chi',\epsilon)$ with even $\chi$ and $\epsilon_p=\pm 1$ for $p=2$ or $p\mid M$ with $\chi\neq 1$, we define the associated modular form space $\mathrm{A}^{\epsilon}(N,k,\chi')$ to be the common eigenspace with eigenvalues $\mu_2^+$ for $Y(4)$, $\mu_p^{\epsilon_p}$ for $Y(p)$ if $\chi_p\neq 1$ and $\mu_p^+=-1$ for $Y(p)$ if $\chi_p=1$. Explicitly, $f=\sum_n a(n)q^n\in A^\epsilon(N,k,\chi')$ if and only if $f|_kY(p)=-f$ if $\chi_p=1$ and $a(n)=0$ whenever $n\equiv 2,-\epsilon_2\imod 4$ or $\left(\frac np\right)=-\epsilon_p$ for some $p\mid M$ with $\chi_p\neq 1$.

\begin{Rmk}
The Kohnen's plus condition for the space $\mathrm{A}(N,k,\chi')$ with $\chi'=\chi\left(\frac N\cdot\right)$ is given by $a(n)=0$ if $\chi_2'(-1)(-1)^{k-\frac 12}n\equiv 2,3\imod 4$. Since $\chi_2'(-1)=\chi_2(-1)\left(\frac{-1}{N}\right)$,
\begin{align*}\epsilon_2\chi_2'(-1)(-1)^{k-\frac 12}=t\chi_2(-1)e_4(2k-1)=\chi_2(-1)e_4(2k-1)e_4(1-t)=\chi_2(-1)e_4(r-t)=1.
\end{align*}
Therefore, the Kohnen's plus condition on $\mathrm{A}(N,k,\chi')$ is the same as our $\epsilon_2$-condition for the data $(N,k,\chi')$ that corresponds to $D$.
This explains partially the twisting by $\left(\frac N\cdot\right)$ in the correspondence $D$ to $(N,\chi',\epsilon)$.
\end{Rmk}

\section{The Isomorphism: Construction and the Proof}

\noindent
Let $N=4M$ with $M$ being positive and odd, and $D_2=2^{+1}_{\pm 1}$. If $\chi$ is a Dirichlet character modulo $N=4M$, we denote $\chi'=\left(\frac{N}{\cdot}\right)\chi$.

\begin{Lem}\label{C-Lem-1}
(1) If $F\in \mathcal A(k,\rho_D)$, then $F_0|\tau_N\in \mathrm{A}(N,k,\chi')$.

(2) If $f\in \mathrm{A}(N,k,\chi')$, then
\[F=\sum_{A\in \Gamma_0(N)\backslash \textrm{SL}_2(\mathbb Z)}\left(f|\tau_N\tilde{A}\right)\rho_D(\tilde A)^{-1}\mathfrak e_0\quad\in \quad\mathcal A(k,\rho_D).\]
\end{Lem}
\begin{proof}
For part (1): By Proposition 1.4 of \cite{shimura1973modular}, we only have to prove that $F_0\in \mathrm{A}(N,k,\chi)$. Let $A\in \Gamma_0(N)$, so
\begin{align*}
F_0|A^* &= \langle F|A^*,\mathfrak e_0\rangle=\langle F|\tilde A(1,\nu(A)),\mathfrak e_0\rangle=\nu(A)^{-2k}\langle F|\tilde A,\mathfrak e_0\rangle=\nu(A)^{-2k}\langle \rho_D(\tilde A)F,\mathfrak e_0\rangle\\
&=\nu(A)^{-2k}\langle F,\rho_D(\tilde A)^{-1}\mathfrak e_0\rangle=\nu(A)^{-2k}\nu(A)^r\chi(A)\langle F,\mathfrak e_0\rangle=\chi(A)F_0,
\end{align*}
which is what we want.

For part (2), note first that $\rho_D(Z)\mathfrak e_\gamma=e(-\frac{r}{4})\mathfrak e_{-\gamma}$, so we have $\rho_D(Z^2)\mathfrak e_\gamma=-\mathfrak e_\gamma$, for any $\gamma$. Let $B\in \Gamma_0(N)$ and $A\in \textrm{SL}_2(\mathbb Z)$. We have
\begin{align*}&\left(f|\tau_N\widetilde{BA}\right)\rho_D(\widetilde{BA})^{-1}\mathfrak e_0=\sigma(B,A)^2\left(f|\tau_N\tilde{B}\tilde{A}\right)\rho_D(\tilde{B}\tilde{A})^{-1}\mathfrak e_0\\
=&\nu(B)^{2k}\left(f|\tau_N B^*\tilde{A}\right)\rho_D(\tilde{A})^{-1}\rho_D(\tilde{B})^{-1}\mathfrak e_0=\left(f|\tau_N \tilde{A}\right)\rho_D(\tilde{A})^{-1}\mathfrak e_0,\end{align*}
since $f|\tau_N\in \mathrm{A}(N,k,\chi)$. Here $(1,\sigma(B,A))\widetilde{BA}=\tilde B\tilde A$. It follows that the sum is independent of the choice of representatives. Now for any $B\in \textrm{SL}_2(\mathbb Z)$,
\begin{align*}
F|\tilde B&=\sum_{A\in \Gamma_0(N)\backslash \textrm{SL}_2(\mathbb Z)}\left(f|\tau_N\tilde A\tilde B\right)\rho_D(\tilde A)^{-1}\mathfrak e_0\\
&=\sigma(A,B)\sigma(AB,B^{-1})\sigma(B,B^{-1})\rho_D(\tilde B)F=\rho_D(\tilde B)F,
\end{align*}
by the cocycle relation of $\sigma$.
Moreover, it is clear that
\[F|Z^2=-F=\rho_D(Z^2)F,\]
so $F\in\mathcal A(k,\rho_D)$ since $\textrm{SL}_2(\mathbb Z)Z^2=\textrm{Mp}_2(\mathbb Z)$.
\end{proof}

We define a quantity which will appear in the isomorphism below. For each integer $n$, define
\[s(n)=\prod_{p: p\mid (M,n)}1+\frac{p}{|D_p|}.\]
It is clear that $s(0)=\sum_{m\mid M}\frac{m}{|D_m|}$, and if $\frac nN$ is a norm and $n'=M/(M,n)$,
\[s(n')=\frac{n'}{|D_{n'}|}\#\left\{\gamma\in D\colon q(\gamma)=\frac nN\right\}.\]
Define a map $\phi_D: \mathcal A^{\textrm{inv}}(k,\rho_D)\rightarrow \mathrm{A}(N,k,\chi')$ by
\[F\mapsto i^{\frac{2k-r}{2}}s(0)^{-1}|D|^{\frac{1}{2}}N^{-\frac{k}{2}}F_0|{\tau_N}.\]
Conversely, we define $\psi_D: \mathrm{A}(N,k,\chi')\rightarrow \mathcal A^{\textrm{inv}}(k,\rho_D)$ by
\[f\mapsto i^{\frac{2k-r}{2}}3^{-1}|D|^{-\frac{1}{2}}N^{\frac{k}{2}}\sum_{A\in\Gamma_0(N)\backslash SL_2(\mathbb Z)}\left(f|\tau_N\tilde{A}\right)\rho_D(\tilde{A})^{-1}\mathfrak e_0.\]
If there is no danger of confusion, we shall drop the subscript and write $\phi$ and $\psi$.

We first prove one side of the isomorphism.
\begin{Lem}\label{C-Lem-2}
We have $\psi\circ\phi=\textrm{id}$.
\end{Lem}
\begin{proof}
We need to prove that for $F\in\mathcal A^{\textrm{inv}}(k,\rho_D)$,
\[\sum_{A\in\Gamma_0(N)\backslash \textrm{SL}_2(\mathbb Z)}F_0|\tilde A\langle \rho_D(\tilde A)^{-1}\frak e_0,\frak e_0\rangle=3s(0)F_0.\]
Let $s$ be any cusp of $\Gamma_0(N)$ and consider the sub-sum
\[F_s=\sum_{A\colon A\infty\sim s}F_0|\tilde A\langle \rho_D(\tilde A)^{-1}\frak e_0,\frak e_0\rangle.\]
It is clear that $F_s$ is $\tilde T$-invariant.
If $s\sim \frac{1}{N/m}$ be a cusp with $m\mid M$, then
\begin{align*}
F_s&=\sum_{j\imod m}F_0|\widetilde{\gamma_mT^j}\langle \rho_D(\widetilde{\gamma_m})^{-1}\frak e_0,\rho_D(\tilde T^j)\frak e_0\rangle\\
&=\sum_{j\imod m}\sum_{\alpha\in D}F_\alpha\langle \rho_D(\widetilde{\gamma_mT^j})\frak e_\alpha,\frak e_0\rangle \langle \rho_D(\widetilde{\gamma_m})^{-1}\frak e_0,\frak e_0\rangle\\
&=\sum_{j\imod m}F_0\langle \rho_D(\widetilde{\gamma_mT^j})\frak e_0,\frak e_0\rangle \langle \rho_D(\widetilde{\gamma_m})^{-1}\frak e_0,\frak e_0\rangle\\
&=\sum_{j\imod m}F_0\langle \rho_D(\widetilde{\gamma_m})\frak e_0,\frak e_0\rangle \langle \rho_D(\widetilde{\gamma_m})^{-1}\frak e_0,\frak e_0\rangle=\frac{m}{|D_m|}F_0,
\end{align*}
where we applied Lemma \ref{T-Lem-4}.
If $s\sim \frac{1}{M/m}$, then by the same computation with $\gamma_{4m}$ in place of $\gamma_m$, we have $F_s=\frac{2m}{|D_m|}F_0$. The cusps $s\sim \frac{1}{2M/m}$ gives $F_s=0$ by the formula in Theorem 6.4 of \cite{stromberg2013weil}. From this, the lemma follows.
\end{proof}

Denote also by $\psi$ when restricted to subspace $\mathrm{A}^\epsilon(N,k,\chi')$. Now we prove that on the subspaces we constructed in the preceding section, $\phi$ and $\psi$ are isomorphisms.

\begin{Thm}\label{Thm} Let $D$ be transitive with $D_2=2^{+1}_t$ and $(N,\chi\left(\frac N\cdot\right),\epsilon)$ correspond to $D$.
The maps $\phi_D$ and $\psi_D$ are inverse isomorphisms between $\mathcal A^\text{inv}(k,\rho_D)$ and  $\mathrm{A}^\epsilon(N,k,\chi\left(\frac N\cdot\right))$. Explicitly, if $f=\sum_na(n)q^n\in \mathrm{A}^\epsilon(N,k,\chi\left(\frac N\cdot\right))$ and $\psi(f)=F=\sum_\gamma F_\gamma \mathfrak e_\gamma$, then
\[F_\gamma(\tau)=\sum_{n\equiv Nq(\gamma)\imod N\mathbb Z}s(n)\frac{M/(M,n)}{|D_{M/(M,n)}|}a(n)q^{\frac{n}{N}}.\]
\end{Thm}
\begin{proof}
To see that both maps are well-defined, we only have to show that the image of each map satisfies the extra condition for each subspace, by Lemma \ref{C-Lem-1}. That $\psi(f)$ is $\textrm{Aut}(D)$-invariant follows from the fact that the action of $\textrm{Mp}_2(\mathbb Z)$ and that of $\textrm{Aut}(D)$ on $\mathbb C[D]$ commute. For $\phi(F)$, since
\[\tau_N=\tilde S\left(\begin{pmatrix}
N&0\\0&1
\end{pmatrix}, N^{-\frac{1}{4}}(-i)^{\frac{1}{2}}\right)\]
and
\[F|\tilde S=\rho_D(\tilde S)F=\sum_{\gamma\in D}F_\gamma \rho_D(\tilde S)\mathfrak e_\gamma=i^{-\frac{r}{2}}|D|^{-\frac{1}{2}}\sum_{\gamma\in D}F_\gamma\sum_{\delta\in D}e(-(\gamma,\delta))\mathfrak e_\delta,\]
we have
\begin{align*}
\phi(F)&=i^{\frac{2k-r}{2}}s(0)^{-1}|D|^{\frac{1}{2}}N^{-\frac{k}{2}}\langle F|\tau_N,\mathfrak e_0\rangle\\
&=i^{\frac{2k-r}{2}}s(0)^{-1}|D|^{\frac{1}{2}}N^{-\frac{k}{2}}N^{\frac{2k}{4}}(-i)^{-\frac{2k}{2}}\langle (F|\tilde{S})(N\tau),\mathfrak e_0\rangle\\
&=i^{\frac{2k-r}{2}}s(0)^{-1}|D|^{\frac{1}{2}}i^{\frac{2k}{2}}i^{-\frac{r}{2}}|D|^{-\frac{1}{2}}\sum_{\gamma\in D} F_\gamma(N\tau)\\
&=s(0)^{-1}\sum_{\gamma\in D} F_\gamma(N\tau):=\sum_nc(n)q^n.
\end{align*}
Now by the action of $T$, the Fourier coefficient $a_\gamma(n)$ of $F_\gamma(N\tau)$ vanish unless $\frac{n}{N}\equiv Q(\gamma)\imod 1$. It follows that $c(n)=0$ unless $\frac{n}{N}\imod 1\in Q(D)$, and by Lemma \ref{D-Lem-3}, this means that $\phi(F)$ satisfies the desired $\epsilon_p$-condition for $p=2$ and for $p\mid M$ with $\chi_p\neq 1$.

When $p\mid M$ with $\chi_p=1$, we need to show that $\phi(F)|Y(p)=-\phi(F)$, that is, $F_0|\tau_NY(p)=-F_0|\tau_N$. By Proposition \ref{A-Prop-1} (1) and (5), we only have to show that
\[\sum_{j\imod p}F_0|\gamma_p^*\tilde{T}^j=-F_0.\]
Indeed,
\begin{align*}
&\sum_{j\imod p}F_0|\gamma_p^*\tilde{T}^j=\nu(\gamma_p)^{-2k}\sum_{j\imod p}F_0|\tilde \gamma_p\tilde{T}^j=\nu(\gamma_p)^{-2k}\sum_{j\imod p}\langle F,\frak e_0\rangle|\tilde \gamma_p\tilde{T}^j\\
=&\nu(\gamma_p)^{-2k}\sum_{j\imod p}\langle F|\tilde\gamma_p,\frak e_0\rangle|\tilde{T}^j=\nu(\gamma_p)^{-2k}\sum_{j\imod p}\langle F,\rho_D(\tilde\gamma_p)^{-1}\frak e_0\rangle|\tilde{T}^j\\
=&\nu(\gamma_p)^{-2k}\sum_{\alpha\in D_p}\sum_{j\imod p}\overline{\xi(\tilde\gamma_p)}p^{-1}F_\alpha|\tilde{T}^j.
\end{align*}
By Lemma \ref{T-Lem-4}, this is equal to $\nu(\gamma_p)^{-2k}\overline{\xi(\tilde\gamma_p)}F_0$. By applying the explicit formula in Theorem 6.4 of \cite{stromberg2013weil}, we have $\nu(\gamma_p)^{-2k}\overline{\xi(\tilde\gamma_p)}=-1$, so $\phi(F)\in \mathrm{A}^\epsilon(N,k,\chi')$.

By Lemma \ref{C-Lem-2}, we are left to prove that $\phi\circ\psi=id$.
For each cusp $s$, we modify above $F_s$ as
\[F_s=\sum_{A\colon A\infty\sim s}\left(f|\tau_N\tilde M\tau_N\right)\langle\rho_D(\tilde M)^{-1}\mathfrak e_0,\mathfrak e_0\rangle,\]
and we need to show that $\sum_s F_s=3s(0)f$.
The computations for $s\sim\frac{1}{N/m}$ and for $s\sim\frac{1}{M/m}$ with $m\mid M$ are similar and all other cusps give $0$. We only have to show that
\[F_s=\frac{m}{|D_m|}\quad\text{and}\quad F_s=\frac{2m}{|D_m|}\]
respectively in the two cases.
Since the former is much easier, we omit it and only treat the case $s\sim\frac{1}{M/m}$. In this case, we have
\begin{align*}
F_s&=\sum_{j\imod 4m}f|\tau_N\widetilde{\gamma_{4m}T^j}\tau_N\langle \rho_D(\widetilde{\gamma_{4m}})^{-1}\frak e_0,\rho_D(\tilde T^j)\frak e_0\rangle\\
&=\sum_{j\imod 4m}f|\tau_N\tilde{S}\widetilde{\gamma_{M/m}}^{-1}\tilde T^j\tau_N\langle \rho_D(\widetilde{\gamma_{4m}})^{-1}\frak e_0,\frak e_0\rangle\\
&=e_8(-2k)\sum_{j\imod 4m}f|\widetilde{\delta_N}^{-1}\widetilde{\gamma_{M/m}}^{-1}\widetilde{T^j}\tau_N\langle \rho_D(\widetilde{\gamma_{4m}})^{-1}\frak e_0,\frak e_0\rangle.
\end{align*}
Assume
\[\gamma_{M/m}=\begin{pmatrix}
a & b\\
c & d
\end{pmatrix},\qquad\beta=
\begin{pmatrix}
dm/M & -b/4m\\
-4mc & aM/m
\end{pmatrix}.\]
Then it is easy to see that $\beta\in\Gamma_0(4)$,
\[N\widetilde{\delta_N}^{-1}\widetilde{\gamma_{M/m}}^{-1}\widetilde{\delta_{4m}}
=4m\tilde{\beta}\widetilde{\delta_{M/m}},\]
and
\[\beta\gamma_{M/m}^{-1}\equiv \left\{
\begin{array}{cl}
\begin{pmatrix}
(M/m)^{-1} & 0\\
0 & M/m
\end{pmatrix} & \imod 4m,\\
\begin{pmatrix}
-(4m)^{-1} & 0\\
0 & -4m
\end{pmatrix} & \imod M/m
\end{array}\right..\]
In particular $\beta\gamma_{M/m}^{-1}\in\Gamma_0(N)$. Observe that $\sigma(\beta\gamma_{M/m}^{-1},\gamma_{M/m})=1$ by the assumption that all of the entries of $\gamma_{M/m}$ are positive. Therefore,
\begin{align*}
F_s&=e_8(-2k)\sum_{j\imod 4m}f|\tilde\beta\widetilde{\delta_{M/m}}\widetilde{\delta_{4m}}^{-1}\widetilde{T^j}\tau_N\langle \rho_D(\widetilde{\gamma_{4m}})^{-1}\frak e_0,\frak e_0\rangle\\
&=e_8(-2k)\sum_{j\imod 4m}f|\widetilde{\beta\gamma_{M/m}^{-1}}\widetilde{\gamma_{M/m}}\widetilde{\delta_{M/m}}\widetilde{\delta_{4m}}^{-1}\widetilde{T^j}\tau_N\langle \rho_D(\widetilde{\gamma_{4m}})^{-1}\frak e_0,\frak e_0\rangle\\
&=e_8(-2k)\nu(\beta\gamma_{M/m}^{-1})^{2k}\chi_{4m}(M/m)\chi_{M/m}(-4m)\\
&\hspace{3cm}\times\sum_{j\imod 4m}f|\widetilde{\gamma_{M/m}}\widetilde{\delta_{M/m}}\widetilde{\delta_{4m}}^{-1}\widetilde{T^j}\tau_N\langle \rho_D(\widetilde{\gamma_{4m}})^{-1}\frak e_0,\frak e_0\rangle\\
&=e_8(-2k)\nu(\beta\gamma_{M/m}^{-1})^{2k}\nu(\gamma_{M/m})^{2k}\chi_{4m}(M/m)\chi_{M/m}(-4m)\\
&\hspace{3cm}\times (4m)^{1-\frac k2}f|W(M/m)U(4m)\tau_N\langle \rho_D(\widetilde{\gamma_{4m}})^{-1}\frak e_0,\frak e_0\rangle\\
&=(4m)^{1-\frac k2}e_8(-2k)\nu(\beta)^{2k}\chi_{4m}(M/m)\chi_{M/m}(-4m)f|W(M/m)U(4m)\tau_N\langle \rho_D(\widetilde{\gamma_{4m}})^{-1}\frak e_0,\frak e_0\rangle.
\end{align*}
By Proposition \ref{A-Prop-1} and the $\epsilon$-condition of $f$, we see that
\begin{align*}(4m)^{1-\frac k2}f|W(M/m)U(4m)\tau_N&=\varepsilon_m^{2k}2^{\frac 32}\left(\chi_2(-1)\left(\frac{-1}{M}\right)\right)^{k+1}(-1)^{\lfloor\frac{2k+1}{4}\rfloor}
\chi_m(-1)\chi_{N/m}(m)
\\ &\qquad\times\left(\frac{M/m}{m}\right)\left(\prod_{p\mid m,\chi_p=1}(-1)\right)\left(\prod_{p\mid m,\chi_p\neq 1}\epsilon_p\varepsilon_p^{-1}p^{\frac 12}\chi_p(m/p)\right)f.
\end{align*}
Moreover,
\[\nu(\beta)^{2k}=\left(\frac{-4mc}{aM/m}\right)\varepsilon_{M/m}^{-2k}
=\left(\frac{c}{d}\right)\left(\frac{m}{M/m}\right)\varepsilon_{M/m}^{2k}.\]
Finally, by Theorem 6.4 of \cite{stromberg2013weil}, we have
\begin{align*}
\langle \rho_D(\widetilde{\gamma_{4m}})^{-1}\frak e_0,\frak e_0\rangle&=|D_{4m}|^{-\frac 12}e_8(2r-t)\chi_{M/m}(-1)\left(\frac{c}{d}\right)\\&\qquad\times\left(\prod_{p\mid m,\chi_p=1}(-1)\right)\left(\prod_{p\mid m,\chi_p\neq 1}\epsilon_p\varepsilon_p^{-1}\chi_p(N/p)\right).
\end{align*}
Let $\eta=-1$ if $\chi_2(-1)=-1$ and $\left(\frac{-1}{M}\right)=-1$ and $\eta=1$ otherwise.
Putting everything together, we have
\[F_s=\frac{2m}{|D_{m}|}e_8(t-2k)\eta\chi_2(M)\chi_2(-1)^k\left(\frac{2k}{2}\right)f
=\frac{2m}{|D_{m}|}f,\]
where by the oddity formula, $\chi_2(-1)=1$ if and only if $2k\equiv t\imod 4$.

Finally, the explicit formula follows easily from the expression of $\phi$ and the fact that for any $p\nmid n$, the number of elements with norm $\frac np$ is $0$ or $p+1$ in $D_p=p^{\pm 1}$. Done.
\end{proof}

Actually, the proof of Theorem \ref{Thm} shows that $\psi_D$ maps other eigenspaces to $0$.

\begin{Cor} \label{C-Cor-1}
Let $f^\epsilon$ be the $\epsilon$-component of $f\in \mathrm{A}(N,k,\chi')$ with respect to the decomposition into common eigen-subspaces. Then $\psi_D(f)=0$ if and only if $f^\epsilon=0$.
\end{Cor}
\begin{proof}
In proof of Theorem \ref{Thm}, if $f\in \mathrm{A}(N,k,\chi')$ such that $f|Y(p)=pf$, then $F_s=-F_{s'}$ when $s\sim \frac{1}{N/m}$ and $s'\sim \frac{1}{N/pm}$ or $s\sim \frac{1}{M/m}$ and $s'\sim \frac{1}{M/pm}$ for each $m\mid M/p$, so $\psi_D(f)=0$. The case when $f|Y(p)=-\epsilon_p\varepsilon_p^{-1}p^\frac 12f$, is similar. If  $f|Y(4)=\mu_2^-f$, then then $F_s=-F_{s'}$ when $s\sim \frac{1}{M/m}$ and $s'\sim \frac{1}{N/m}$ and $\psi_D(f)=0$.
\end{proof}

\begin{Cor} \label{C-Cor-2.5}
The isomorphisms $\phi$ and $\psi$ induces isomorphisms
\[\mathcal M^\text{inv}(k,\rho_D)\simeq \mathrm{M}^\epsilon(N,k,\chi')\quad\text{and}\quad \mathcal S^\text{inv}(k,\rho_D)\simeq \mathrm{S}^\epsilon(N,k,\chi').\]
Consequently, $f\in A^\epsilon(N,k,\chi')$ is holomorphic (resp. cuspidal) if and only if $f$ is holomorphic (resp. vanishing) at $\infty$.
\end{Cor}
\begin{proof}
Let $F=\sum_\gamma F_\gamma\mathfrak e_\gamma\in \mathcal A^\textrm{inv}(k,\rho_D)$. Since for each $A\in\textrm{SL}_2(\mathbb Z)$, $F|\tilde{A}=\sum_\gamma F_\gamma\rho_D(\tilde A)\mathfrak e_\gamma$,
$F$ is holomorphic (resp. cuspidal) if and only if $F_\gamma$ is holomorphic (resp. vanishing) at $\infty$ for each $\gamma$. By Lemma \ref{T-Lem-4}, this is equivalent to saying that $F_0|\tilde{A}$ is holomorphic (resp. vanishing) at $\infty$ for all $A\in \textrm{SL}_2(\mathbb Z)$, that is $F_0$ is holomorphic (resp. cuspidal). This in turn is equivalent to saying that $\phi(F)$ is holomorphic (resp. cuspidal) since $\tau_N$ preserves the holomorphic subspace and the cuspidal subspace.
\end{proof}

We shall focus on the simpler situation when $\chi_p\neq 1$ for all $p\mid M$.

\begin{Cor} \label{C-Cor-2}
Assume that $D$ and $(N,k,\chi')$ correspond and $\chi_p\neq 1$ for all $p\mid M$. Let $f=\sum_na_n(y)e(nx)$ is a real analytic modular form of level $N$, weight $k$, character $\chi'$ that satisfies the $\epsilon$-condition, and $F=\psi(f)$. Then $\phi(F)=f$ and
\[F_\gamma(\tau)=\sum_{n\equiv Nq(\gamma)\imod N\mathbb Z}s(n)a_n(y/N)e(nx/N),\]
and $s(n)$ is equal to the number of distinct positive divisors of $(n,M)$.
\end{Cor}
\begin{proof}
The proofs above on $\phi\circ\psi=\textrm{id}$ and $\psi\circ\phi=\textrm{id}$ and that of Lemma \ref{T-Lem-4} are independent of whether $f$ or $F$ are holomorphic.
The assumption implies that $m=|D_m|$ for each $m\mid M$ and the corollary follows.
\end{proof}

\section{Zagier Duality and Borcherds' Theorem}

From now on, we shall assume that $\chi_p\neq 1$ for each $p\mid M$, so $\chi'=1$. We extend the notion of reduced modular forms in \cite{zhang2015zagier} to current setting: $f\in \mathrm{A}^\epsilon(N,k,\chi')$ is called reduced if $f=\frac{1}{s(m)}q^{m}+O(q^{m+1})$ for some integer $m$ and if for each $n>m$ with $a(n)\neq 0$, there does not exist $g\in \mathrm{A}^\epsilon(N, k, \chi')$ such that $g=q^n+O(q^{n+1})$. If it exists, it must be unique and $\chi_p(m)\neq -\epsilon_p$ for each $p\mid M$; we denote it by $f_m$. It is also clear that the set of reduced modular forms is a basis for $\mathrm{A}^\epsilon(N,k,\chi')$.

We first consider the existence of $f_m$ for $m<0$. Let $D^*$ be the dual discriminant form of $D$ given by the same abelian group with discriminant form $-q(\cdot)$. It is clear that $D^*$ is also transitive and the corresponding data is $(N,\chi',\epsilon^*)$ with $\epsilon_p^*=\chi_p(-1)\epsilon_p$.

\begin{Prop}
Let $B^*=\{m\colon f_m^*\in \mathrm{M}^{\epsilon^*}(N, 2-k, \chi')\text{ exists}\}$. Then for any $m<0$ with $\chi_p(m)\neq -\epsilon_p$ for all $p\mid M$, we have $f_m\in \mathrm{A}^{\epsilon}(N, k, \chi')$ exists if and only if $-m\notin B^*$.
\end{Prop}
\begin{proof}
The obstruction theorem, Theorem 3.1, of \cite{borcherds1999gross} implies the following: let $P=\sum_{n\leq 0}a(n)q^n$ be a polynomial in $q^{-1}$ with $a(n)=0$ if $\chi_p(n)=-\epsilon_p$ for some $p\mid M$ or $n\equiv 2,-\epsilon_2\imod 4$. Then there exists $f\in \mathrm{A}^\epsilon(N,k,\chi')$ with $f=\sum_na(n)q^n$ if and only if
\[\sum_{n\leq 0}s(n)a(n)b(-n)=0\]
for each $g=\sum_n b(n)q^n\in \mathrm{M}^{\epsilon^*}(N,2-k,\chi')$.

If $-m\in B^*$, then by the obstruction of $f_{-m}^*$, $f_m$ does not exist. Conversely, if $-m\notin B^*$, assume $B^*=\{n_i\}$ and $f_{n_i}^*=\sum_na_i(n)q^n$, so $s(n_i)a_j(n_i)=\delta_{ij}$. Let
\[P=\frac{1}{s(m)}q^m-\frac{1}{s(m)}\sum_i s(n_i)a_i(-m)q^{-n_i},\]
and we see that $P$ satisfies the obstruction linear system, so the existence of $f_m$ follows.
\end{proof}

Now we prove the Zagier duality.

\begin{Thm}\label{R-Thm-1}
Let $m,d$ be integers and assume that both of the reduced modular forms
\begin{align*}
f_m&=\sum_n a_m(n)q^n \in \mathrm{A}^{\epsilon}(N, k, \chi')\\
f_d^*&=\sum_n a_d^*(n)q^n \in \mathrm{A}^{\epsilon^*}(N, 2-k, \chi')
\end{align*}
exist. Then $a_m(-d)=-a_d^*(-m)$.
\end{Thm}
\begin{proof} The statement is trivial when $m>0$ and $d>0$ or $m=0$ and $d>0$ or $m>0$ and $d=0$, so by symmetry we may assume that $m=d=0$ or $m<0$.

Let $F=\psi(f_m)$ and $G=\psi(f_d^*)$. It is clear that $H=\sum_\gamma F_\gamma G_\gamma$ is a weakly holomorphic modular form of weight $2$ for $\textrm{SL}_2(\mathbb Z)$. Therefore, the sum of the residues of the meromorphic $1$-form $H(\tau)d\tau$ on the compact Riemann surface $X(1)$ vanishes. Since $F$ and $G$ are holomorphic on $\mathbb H$ and the residue of $H(\tau)d\tau$ at $\infty$ is given by $\frac{s(0)}{2\pi i}$ times $\sum_{n\in\mathbb Z}s(n)a_m(n)a_d^*(-n)$. If $(m,d)=(0,0)$, we have a contradiction, so $f_0$ and $f_0^*$ cannot both exist. We then assume $m<0$, we have
\[\sum_{n\in\mathbb Z}s(n)a_m(n)a_d^*(-n)=a_m(-d)+a_d^*(-m)+\sum_{m<n<-d}s(n)a_m(n)a_d^*(-n)=0.\]
So we only have to prove that $\sum_{m<n<-d}s(n)a_m(n)a_d^*(-n)=0$. If $m<n\leq 0$, then $a_m(n)=0$ if $-n\in B^*$ and $a_d^*(-n)=0$ if $-n\notin B^*$. Similarly, if $0<n<-d$ and $a_d^*(-n)\neq 0$, then $-n\notin B^*$ and $f_n$ exists, so $a_m(n)=0$. We are done.
\end{proof}

In the rest of this section, we write down explicitly the Borcherds lift in the case of $\mathrm{O}(2,1)$. We consider the following even lattice
\[L=\left\{\begin{pmatrix}
a & b/M\\
c & -a
\end{pmatrix}\colon a,b,c\in\mathbb Z\right\},\]
with $q(\alpha)=-M\textrm{det}(\alpha)$ and $(\alpha,\beta)=M\textrm{tr}(\alpha\beta)$. Then the discriminant form $D=L^\vee/L\simeq \mathbb Z/2M\mathbb Z$ with $D=\prod_{p\mid 2M}D_p$ given by
\[D_2=2^{+1}_t, t=\left(\frac{-1}{M}\right),\quad D_p=p^{\delta_p}, \delta_p=\left(\frac{2M/p}{p}\right), p\mid M.\]
It follows that for such $D$, $\epsilon_p=+1$ for all $p\mid M$ and $\chi'=\chi\left(\frac N\cdot\right)=1$, so we shall simply denote
$\mathrm{A}^+(N,k,1)$ for $\mathrm{A}^\epsilon(N,k,\chi')$. The dual $D^*$ of $D$ then gives $\epsilon^*$ with $\epsilon^*_p=\chi_p(-1)$.

Let us recall Zagier's non-holomorphic modular form
\[\mathbf G(\tau)=\sum_{n=0}^{\infty}H(n)q^n+\frac{1}{16\pi}\sum_{n\in\mathbb Z}q^{-n^2}\int_y^\infty e^{-4\pi un^2}u^{-\frac 32}du\]
for $\Gamma_0(4)$ of weight $\frac 32$ (see \cite{zagier1975nombres}). Here $H(n)$ denotes the Hurwitz class number of $n$, whose generating function, that is the the holomorphic part of $\mathbf G(\tau)$, will be denoted by $G(\tau)=\sum_{n=0}^{\infty}H(n)q^n$. Denote $\mathbf G^{*}$ the $\epsilon^*$-component of $\mathbf G$ and denote the its holomorphic part by $G^{*}(\tau)=\sum_{n=0}^{\infty}H^{*}(n)q^n$.
The map $\psi$ in the preceding section can be extended to non-holomorphic modular forms. In particular, we can consider $\mathbf G(\tau)$ as in $\mathrm{A}(N,\frac 32,1)$, so $\psi_{D^*}(\mathbf G)$ is a non-holomorphic modular form of weight $\frac 32$ and type $\rho_D$ by Corollary \ref{C-Cor-2}. We omit the dependence on $D$ and denote it by $\mathbf G_N(\tau)=\psi_{D^*}(\mathbf G)$. By Corollary \ref{A-Cor-1} and that $H(n)\neq 0$ for any positive $n\equiv 0,3\imod 4$, we see that $\mathbf G_N(\tau)$ is non-zero. We denote the holomorphic part of $\mathbf G_N(\tau)$ by $G_N(\tau)$.

\begin{Thm}\label{Thm2}[Borcherds] Let $D$ and $(N,\chi\left(\frac N\cdot\right),\epsilon)$ be as above. Assume $f=\sum_nc(n)q^n\in \mathrm{A}^+(N,k,1)$ with $s(n)c(n)\in\mathbb Z$ for all $n\leq 0$. Then there exists a meromorphic modular form $\Psi(f)$ of weight $s(0)c(0)$ for $\Gamma_0(M)$ (with some finite multiplier system) such that

(1) $\Psi(f)$ has an infinite product expression:
\[\Psi(f)(\tau)=q^{\rho}\prod_{n=1}^{\infty}(1-q^n)^{s(n^2)c(n^2)},\]
where $\rho=-\sum_{n\in\mathbb Z}s(n)c(-n)H^{*}(n)$.

(2) The zeros and poles of $\Psi(f)$ on $\mathbb H$ occur at CM points $\tau$ of discriminant $D$ with order
\[\sum_{n=1}^\infty s(Dn^2)c(Dn^2).\]
\end{Thm}
\begin{proof}
In Theorem 13.3 of \cite{borcherds1998automorphic},
\[z=\begin{pmatrix}
0&0\\1&0
\end{pmatrix},\quad z'=\begin{pmatrix}
0&1/M\\0&0
\end{pmatrix},\quad K\simeq \mathbb Z \text{ with } q(n)=n^2.\]
Let $F=\psi(f)=\sum_{\gamma}F_\gamma\mathfrak e_\gamma$ and by Corollary \ref{C-Cor-2},
\[F_\gamma(\tau)=\sum_{n\equiv Nq(\gamma)\imod N\mathbb Z}s(n)c(n)q^{\frac{n}{N}}.\]
By Theorem 13.3 of \cite{borcherds1998automorphic}, the Borcherds lift of $F$ is equal to
\[\Psi(f):=\Psi(F)(\tau)=q^{\rho}\prod_{n=1}^{\infty}(1-q^n)^{s(n^2)c(n^2)},\]
where $\rho$ comes from the Weyl vector.

For the formula of $\rho$, we first note that $\mathbf G_N$ satisfies Lemma 9.5 of \cite{borcherds1998automorphic}. We caution here that such properties of $\mathbf G_N$ do not determine it uniquely and it is the proof of that lemma that describes it uniquely; that is, the same map sends $\theta$ to $\Theta_{2M}$ and $\mathbf G$ to $\mathbf G_N$. By comparing the $\mathfrak e_0$-components and Lemma \ref{T-Lem-4}, we see that $\Theta_{2M}=s(0)^{-1}\psi(\theta)$. So by Corollary 9.6 of \cite{borcherds1998automorphic}, $\rho$ is equal to the constant term of $-s(0)^{-1}\langle F,\overline{G_N}\rangle$.

That $\Psi(f)$ is a modular form for $\Gamma_0(M)$ follows from $\Gamma_0(M)\subset \mathrm{O}^+(L)$, where for each $\alpha\in\Gamma_0(M)$, $\beta\mapsto \alpha\beta\alpha^{-1}$ gives the embedding.

By Theorem 13.3 of \cite{borcherds1998automorphic}, the divisor of $\Psi(f)$ on $\mathbb H$ is given by
\[\sum_{\lambda\in L'/\{\pm 1\}, q(\lambda)<0}s(Nq(\lambda))c(Nq(\lambda))T_\lambda,\]
where for $\lambda=\begin{pmatrix}
a/2M&b/M\\c&-a/2M
\end{pmatrix}$, $T_\lambda$ is the unique solution on $\mathbb H$ of the equation \[Mb\tau^2-a\tau-c=0.\] Therefore, $T_\lambda=\tau$ is a CM point and the formula follows. We are done with the proof.
\end{proof}

For computational purpose, we include the following proposition.

\begin{Prop}
We have $\mathbf G-\mathbf G^{*}\in \mathrm{M}(N,\frac 32,1)$.
\end{Prop}
\begin{proof}
By the structure theory of weak Maass forms (see \cite{bruinier2003geometric}), the space $\mathrm{H}^{\geq 0}/\mathrm{A}(N,\frac 32,1)\simeq \mathrm{M}(N,\frac 12,1)$, where $\mathrm{H}^{\geq 0}$ denote the subspace of weak Maass forms $f=f^++f^-$ for $\Gamma_0(N)$ of weight $\frac 32$ with $f^-$ of polynomial growth at $\infty$. By Theorem A in \cite{Serre1977Modular}, $\mathrm{M}(N,\frac 12,1)$ has dimension $1$ and is generated by the Jacobi theta function $\theta(\tau)=\sum_nq^{n^2}$. Since the holomorphic part of $\mathbf G$ is holomorphic at cusps, so is that of $\mathbf G^{*}$. The statement follows.
\end{proof}

\begin{Rmk}
For simplicity, in the above treatment, we only considered the case $\chi_p\neq 1$ for each $p\mid M$. To make full use of the isomorphism in Theorem \ref{Thm}, we can consider the following lattice
\[L=\left\{\begin{pmatrix}
a & b/M_1\\
cM_2 & -a
\end{pmatrix}\colon a,b,c\in\mathbb Z\right\},\]
with $q(\alpha)=-M_1\textrm{det}(\alpha)$ and $(\alpha,\beta)=M_1\textrm{tr}(\alpha\beta)$. Now for $p\mid M_2$, $\chi_p=1$. Then case when $M_1=1$ for the Shimura lift is considered in \cite{li2015shimura}.
\end{Rmk}

\section{Some Examples}

The simplest case $N=4$ has been explored extensively by many people. The Zagier duality is worked out by Zagier in order to prove Borcherds' theorem. Jacobi theta function $\theta=1+\sum_{n=1}^{\infty}q^{n^2}\in \mathrm{M}^+(4,\frac{1}{2},1)$ has Borcherds lift $\eta^2(\tau)$ or that of $12\theta$ is $\Delta(\tau)=\eta^{24}(\tau)$.

Let us consider the case $N=12$ and provide some new examples. Using MAGMA, we record the following reduced modular forms for $\mathrm{A}^+(12,\frac 12,1)$:
\begin{align*}
f_0&=\frac 12 + q + q^4 + q^9 +O(q^{16})\\
f_{-3}&=\frac 12 q^{-3} - 7q + 20q^4 - 39q^9 + 84q^{12} - 189q^{13}  + O(q^{16})\\
f_{-8}&=q^{-8} - 34q - 188q^4 + 2430q^9 + 8262q^{12} - 11968q^{13}  +O(q^{16})\\
f_{-11}&=q^{-11} + 22q - 552q^4 - 11178q^9 + 48600q^{12} + 76175q^{13} +O(q^{16}).
\end{align*}
Similarly, for $\mathrm{A}^{\epsilon^*}(12,\frac 32,1)$:
\begin{align*}
f_{-1}^*&=q^{-1}  -1 + 7q^3 + 34q^8 - 22q^{11} - 26q^{12}+O(q^{15})\\
f_{-4}^*&=q^{-4}  -1 - 20q^3 + 188q^8 + 552q^{11} - 701q^{12} + O(q^{15})\\
f_{-9}^*&=\frac 12 q^{-9}  -1 + 39q^3 - 2430q^8 + 11178q^{11} - 8826q^{12}+ O(q^{15})\\
f_{-12}^*&=\frac 12 q^{-12} - 84q^3 - 8262q^8 - 48600q^{11} - 41412q^{12} + O(q^{15}).\\
\end{align*}
To continue the lists, we simply multiply $j(12\tau)$ and subtract the existing forms. The Zagier duality is clear from these two lists.

To see the Borcherds products, we have a basis for $\mathrm{M}(12,\frac 32,1)$:
\begin{align*}
g_0&=1 + 2q^3 + 6q^4 + 12q^7 + O(q^{12})\\
g_1&=q + q^3 + 2q^4 + 2q^6 + 2q^7 + q^9 + 4q^{10} + O(q^{12})\\
g_2&=q^2 - q^4 + 2q^5 + q^6 - 2q^7 + q^8 + 2q^9 + 2q^{11} + O(q^{12}).
\end{align*}
Since
\[G=-\frac{1}{12}+\frac 13q^3+\frac 12q^4+q^7+q^8+q^{11}+ O(q^{12}),\]
we must have
\[G^{\epsilon^*}=G-\frac{1}{12}g_0=-\frac{1}{6}+\frac 16q^3+q^8+q^{11}+ O(q^{12}).\]
For reduced modular forms in $\mathrm{A}^+(12,\frac 12,1)$, it is clear that $f_0=\frac 12\theta$. It is easy to see that $\rho=\frac 16$ and
\[\Psi(f_0)=q^{\frac 16}\prod_{n=1}^{\infty}(1-q^n)^{s(n^2)}=q^{\frac 16}\prod_{n=1}^{\infty}(1-q^n)(1-q^{3n})=\eta(\tau)\eta(3\tau)\]
is a modular form of weight $1$ for $\Gamma_0(3)$ of finite multiplier system and holomorphic and non-vanishing on $\mathbb H$. Similarly,
\[\Psi(f_{-3})=q^{-\frac 16}\prod_{n=1}^{\infty}(1-q^n)^{s(n^2)c(n^2)}=q^{-\frac 16}(1-q)^{-7}(1-q^{2})^{20}(1-q^{3})^{-78}(1-q^{4})^{344}\cdots\]
is a weakly holomorphic modular form of weight $0$ for $\Gamma_0(3)$ which has simple poles at cusps and simple zeros at CM-points of discriminant $-3$. More explicitly, let
\[E_1=1 + 6q + 6q^3 + 6q^4 + 12q^7 + 6q^9 + O(q^{12})\]
be the unique modular form of weight $1$, level $3$ and character $\left(\frac\cdot 3\right)$ with leading coefficient $1$, then
\[\Psi(f_{-3})=E_1(\tau)\eta(\tau)^{-1}\eta(3\tau)^{-1}.\] In other words, $\Psi(f_{-3}+f_{0})=E_1$. One can also work out the Borcherds product for more reduced modular forms in the same way.

\vskip 0.5 cm

\addcontentsline{toc}{chapter}{Bibliography}
\bibliographystyle{amsplain}
\bibliography{C:/Users/Administrator/OneDrive/Math/paper}

\end{document}